\documentclass[11pt,reqno,letterpaper]{amsart}

\usepackage{amssymb,amsmath,amsthm,amsfonts}
\usepackage{bbm,enumerate}
\usepackage{mathtools} 

\usepackage{bookmark}
\usepackage{hyperref}
\hypersetup{pdfstartview={FitH}}

\theoremstyle{plain}
\newtheorem{theorem}{Theorem}
\newtheorem{problem}[theorem]{Problem}
\newtheorem{lemma}[theorem]{Lemma}
\newtheorem{corollary}[theorem]{Corollary}

\theoremstyle{definition}

\theoremstyle{remark}
\newtheorem{remark}{Remark}

\numberwithin{equation}{section}

\renewcommand{\leq}{\leqslant}
\renewcommand{\geq}{\geqslant}

\renewcommand{\ge}{\geqslant}

\newcommand{\R}{\mathbb{R}}
\newcommand{\N}{\mathbb{N}}
\newcommand{\eps}{\varepsilon}
\newcommand{\1}{\mathbbm{1}}

\DeclareFontFamily{U}{mathx}{\hyphenchar\font45}
\DeclareFontShape{U}{mathx}{m}{n}{
<5> <6> <7> <8> <9> <10>
<10.95> <12> <14.4> <17.28> <20.74> <24.88>
mathx10
}{}
\DeclareSymbolFont{mathx}{U}{mathx}{m}{n}
\DeclareFontSubstitution{U}{mathx}{m}{n}
\DeclareMathAccent{\widecheck}{0}{mathx}{"71}
\DeclareMathAccent{\wideparen}{0}{mathx}{"75}

\begin{document}

\title{Multi-parameter maximal Fourier restriction}

\author[A. Bulj]{Aleksandar Bulj}
\author[V. Kova\v{c}]{Vjekoslav Kova\v{c}}

\address{Department of Mathematics, Faculty of Science, University of Zagreb, Bijeni\v{c}ka cesta 30, 10000 Zagreb, Croatia}
\email{aleksandar.bulj@math.hr}
\email{vjekovac@math.hr}

\subjclass[2020]{Primary 42B10; 
Secondary 42B25, 
37L50} 

\keywords{Fourier transform, Fourier restriction operator, maximal estimate, multi-parameter estimate, convergence almost everywhere, Christ--Kiselev lemma}

\begin{abstract}
The main result of this note is the strengthening of a quite arbitrary a priori Fourier restriction estimate to a multi-parameter maximal estimate of the same type. 
This allows us to discuss a certain multi-parameter Lebesgue point property of Fourier transforms, which replaces Euclidean balls by ellipsoids. 
Along the lines of the same proof, we also establish a $d$-parameter Menshov--Paley--Zygmund-type theorem for the Fourier transform on $\mathbb{R}^d$.
Such a result is interesting for $d\geq2$ because, in a sharp contrast with the one-dimensional case, the corresponding endpoint $\textup{L}^2$ estimate (i.e., a Carleson-type theorem) is known to fail since the work of C. Fefferman in 1970.
Finally, we show that a Strichartz estimate for a given homogeneous constant-coefficient linear dispersive PDE can sometimes be strengthened to a certain pseudo-differential version. 
\end{abstract}

\maketitle


\section{Introduction}
A classical sub-branch of harmonic analysis, started in the late 1960s, asks to restrict meaningfully the Fourier transform $\widehat{f}$ of a certain non-integrable function $f$ to certain curved lower-dimensional subsets of the Euclidean space; see Stein's book \cite[\S{}VIII.4]{Ste93}.
A general setting is obtained by taking a $\sigma$-finite measure $\sigma$ on Borel subsets of $\mathbb{R}^d$. Also, let $S\subseteq\mathbb{R}^d$ be a Borel set such that $\sigma(\mathbb{R}^d\setminus S)=0$. 
Typically, $S$ is a closed manifold in $\mathbb{R}^d$ and $\sigma$ is an appropriately weighted surface measure on $S$.
As soon as we have an a priori estimate
\begin{equation}\label{eq:restr_apriori}
\big\|\widehat{f}\,\big|_S\big\|_{\textup{L}^q(S,\sigma)} \lesssim_{d,\sigma,p,q} \|f\|_{\textup{L}^p(\mathbb{R}^d)}
\end{equation}
for some $p\in(1,\infty)$ and $q\in[1,\infty]$, we can define the \emph{Fourier restriction operator} as the unique bounded linear operator
\begin{equation*}
\mathcal{R}\colon\textup{L}^p(\mathbb{R}^d)\to\textup{L}^q(S,\sigma)
\end{equation*}
such that
$\mathcal{R}f=\widehat{f}|_S$ for every function $f$ in the Schwartz space $\mathcal{S}(\mathbb{R}^d)$.
Here and in what follows, we write $A\lesssim_P B$, when the estimate $A\leq C_P B$ holds for some finite (but unimportant) constant $C_P$ depending on a set of parameters $P$.

Let us agree to use the following normalization of the Fourier transform:
\[ (\mathcal{F}f)(\xi) = \widehat{f}(\xi) := \int_{\mathbb{R}^d} f(x) e^{-2\pi\mathbbm{i}x\cdot\xi} \,\textup{d}x \]
for an integrable function $f$ on $\mathbb{R}^d$ and for every $\xi\in\mathbb{R}^d$, so that the inverse Fourier transform is given by
\[ \widecheck{g}(x) := \int_{\mathbb{R}^d} g(\xi) e^{2\pi\mathbbm{i}x\cdot\xi} \,\textup{d}\xi \]
for $g\in\textup{L}^1(\mathbb{R}^d)$ and $x\in\mathbb{R}^d$.
We always have the trivial estimate
\begin{equation}\label{eq:restr_apriori2}
\big\|\widehat{f}\,\big|_S\big\|_{\textup{L}^\infty(S,\sigma)} \leq \|f\|_{\textup{L}^1(\mathbb{R}^d)}
\end{equation}
for every $f\in\textup{L}^1(\mathbb{R}^d)$, so restriction of the Fourier transform $f\mapsto\widehat{f}|_S$ also gives a bounded linear operator
\begin{equation*}
\mathcal{R}\colon\textup{L}^1(\mathbb{R}^d)\to\textup{L}^\infty(S,\sigma).
\end{equation*}
Using the Riesz--Thorin theorem to interpolate between \eqref{eq:restr_apriori} and \eqref{eq:restr_apriori2} then gives us a family of bounded linear operators
\begin{equation*}
\mathcal{R}\colon\textup{L}^s(\mathbb{R}^d)\to\textup{L}^{qs'/p'}(S,\sigma)
\end{equation*}
for every $1\leq s\leq p$, where $p'$ denotes the conjugated exponent of $p$, i.e., $1/p+1/p'=1$.
All these operators are mutually compatible on their intersections, so they are rightfully denoted by the same letter $\mathcal{R}$.

A novel route was taken recently by M\"{u}ller,  Ricci, and Wright \cite{MRW19}, who initiated the program of justifying pointwise Fourier restriction, 
\[ \lim_{t\to0+} \widehat{f}\ast\chi_{t} = \mathcal{R}f \quad \text{$\sigma$-a.e.\@ on $S$} \]
for $f\in\textup{L}^p(\mathbb{R}^d)$,
via maximal estimates
\begin{equation}\label{eq:restr_maximal}
\Big\|\sup_{t\in(0,\infty)}\big|\widehat{f}\ast\chi_{t}\big|\Big\|_{\textup{L}^q(S,\sigma)} \lesssim_{d,\sigma,\chi,p,q} \|f\|_{\textup{L}^p(\mathbb{R}^d)}.
\end{equation}
Here, $\chi\in\mathcal{S}(\mathbb{R}^d)$ is a Schwartz function with integral $1$ and we write $\chi_t(x):=t^{-d}\chi(t^{-1}x)$ for a given parameter $t\in(0,\infty)$.
Note that the operator on left hand side of \eqref{eq:restr_maximal} cannot be understood as a composition of the Fourier transform with some maximal function of the Hardy--Littlewood type, since the measure $\sigma$ can be (and typically is) singular with respect to the Lebesgue measure.

The authors of \cite{MRW19} achieved the aforementioned goal in two dimensions by imitating the proofs of (somewhat definite) two-dimensional restriction theorems of Carleson and Sj\"{o}lin \cite{CS72} and Sj\"{o}lin \cite{S74}. This methodology was later followed by Ramos \cite{Ram20,Ram22}, Jesurum \cite{Jes22}, and Fraccaroli \cite{Fra21} to obtain a few higher-dimensional or less smooth/regular results.
The second approach to the maximal Fourier restriction was suggested by Vitturi \cite{Vit17}, soon after the appearance of \cite{MRW19}. He deduced a nontrivial result for higher-dimensional compact hypersurfaces from ordinary restriction estimates \eqref{eq:restr_apriori} by inserting the iterated Hardy--Littlewood maximal function in a clever non-obvious way. The idea of using \eqref{eq:restr_apriori} as a black box was later also employed by Oliveira e Silva and one of the present authors \cite{KOeS21}, while the subsequent paper \cite{Kov19} built on this idea to show that the a priori estimate \eqref{eq:restr_apriori} implies the maximal estimate \eqref{eq:restr_maximal} in a general and abstract way, as soon as $p<q$.
Each of these two approaches has its advantages and its limitations. The present paper builds further upon the second approach and it has been partially motivated by a post on Vitturi's blog \cite{Vit19}.
In fact, Theorem~\ref{thm:meainrestr} below answers one of the open questions that appeared in \cite[\S4]{Vit19}. 

\smallskip
For a given function $\chi\colon\mathbb{R}^d\to\mathbb{C}$ and arbitrary parameters $r_1,\ldots,r_d\in(0,\infty)$ we define the \emph{multi-parameter dilate} of $\chi$ as
\[ \chi_{r_1,\ldots,r_d}\colon\mathbb{R}^d\to\mathbb{C}, \quad \chi_{r_1,\ldots,r_d}(x_1,\ldots,x_d) := \frac{1}{r_1\cdots r_d} \chi\Big(\frac{x_1}{r_1},\ldots,\frac{x_d}{r_d}\Big). \]
Also let
\[ B_{r_1,\ldots,r_d}(y_1,\ldots,y_d) := \bigg\{ (x_1,\ldots,x_d)\in\mathbb{R}^d : \frac{(x_1-y_1)^2}{r_1^2} + \cdots + \frac{(x_d-y_d)^2}{r_d^2} \leq 1 \bigg\} \]
be the \emph{ellipsoid} centered at $(y_1,\ldots,y_d)\in\mathbb{R}^d$ with semi-axes of lengths $r_1,\ldots,r_d$ in directions of the coordinate axes.
Its volume will be written simply as $|B_{r_1,\ldots,r_d}|$.
The particular case $B_r(y) := B_{r,\ldots,r}(y)$ for $r\in(0,\infty)$ is simply the Euclidean ball.

\begin{theorem}\label{thm:meainrestr}
Suppose that the measure space $(S,\sigma)$ and the exponents $1<p<q<\infty$ are such that the a priori restriction estimate \eqref{eq:restr_apriori} holds for every Schwartz function $f$.
\begin{itemize}
\item[(a)] Then for every $\chi\in\mathcal{S}(\mathbb{R}^d)$ and every $f\in\textup{L}^p(\mathbb{R}^d)$ one also has the multi-parameter maximal estimate
\begin{equation}\label{eq:restr_multimax}
\Big\|\sup_{r_1,\ldots,r_d\in(0,\infty)}\big|\widehat{f}\ast\chi_{r_1,\ldots,r_d}\big|\Big\|_{\textup{L}^q(S,\sigma)} \lesssim_{d,\sigma,\chi,p,q} \|f\|_{\textup{L}^p(\mathbb{R}^d)}.
\end{equation}
\item[(b)] For every $\chi\in\mathcal{S}(\mathbb{R}^d)$ such that $\int_{\mathbb{R}^d}\chi=1$ and every $f\in\textup{L}^s(\mathbb{R}^d)$, $1\leq s\leq p$, one also has the multi-parameter convergence result
\begin{equation}\label{eq:restr_convergence}
\lim_{(0,\infty)^d\ni(r_1,\ldots,r_d)\to(0,\ldots,0)} \widehat{f}\ast\chi_{r_1,\ldots,r_d} = \mathcal{R}f \quad \text{$\sigma$-a.e.\@ on $S$}.
\end{equation}
\item[(c)] Moreover, if $f\in\textup{L}^{s}(\mathbb{R}^d)$, $1\leq s\leq 2p/(p+1)$, then we also have the ``multi-parameter Lebesgue point property''
\begin{equation}\label{eq:restr_Lebesgue}
\lim_{(0,\infty)^d\ni(r_1,\ldots,r_d)\to(0,\ldots,0)} \frac{1}{|B_{r_1,\ldots,r_d}|} \int_{B_{r_1,\ldots,r_d}(\xi)} \big| \widehat{f}(\eta) - (\mathcal{R}f)(\xi) \big| \,\textup{d}\eta = 0
\end{equation}
of $\sigma$-almost every point $\xi\in S$. In particular,
\begin{equation}\label{eq:restr_ellipsoid}
\lim_{(0,\infty)^d\ni(r_1,\ldots,r_d)\to(0,\ldots,0)} \frac{1}{|B_{r_1,\ldots,r_d}|} \int_{B_{r_1,\ldots,r_d}(\xi)} \widehat{f}(\eta) \,\textup{d}\eta = (\mathcal{R}f)(\xi)
\end{equation}
for $\sigma$-a.e.\@ $\xi\in S$.
\end{itemize}
\end{theorem}

Since \eqref{eq:restr_multimax} is a stronger maximal inequality than \eqref{eq:restr_maximal}, Theorem~\ref{thm:meainrestr} can be viewed as a multi-parameter generalization of \cite[Theorem~1]{Kov19} suggested by Vitturi \cite[\S4]{Vit19}.
For instance, by \eqref{eq:restr_convergence} now we are able to justify the existence of limits in various anisotropic scalings, such as
\[ \lim_{t\to0+} \widehat{f}\ast\chi_{t,t^2,\ldots,t^d}. \]
However, the required assumptions on $\chi$ are more restrictive here.
The proof of Theorem~\ref{thm:meainrestr} will only use that $\chi$ is a function satisfying
\begin{equation}\label{eq:chi_cond}
\big|\big(\partial_1\cdots\partial_d\widehat{\chi}\big)(x)\big| \lesssim_{d,\delta} (1+|x|)^{-d-\delta}
\end{equation}
for some $\delta>0$ and every $x\in\mathbb{R}^d$.
The last condition is different from
\[ \big|\big(\nabla\widehat{\chi}\big)(x)\big| \lesssim_{d,\delta} (1+|x|)^{-1-\delta}, \]
used in \cite{Kov19}, and \eqref{eq:chi_cond} is not satisfied when $\chi$ is the normalized indicator function of the standard unit ball in $d\geq2$ dimensions.
This is the reason why we conclude the convergence of the Fourier averages over shrinking ellipsoids \eqref{eq:restr_ellipsoid} only in the smaller range $1\leq s\leq 2p/(p+1)$ and not in the full range $1\leq s\leq p$, as it was the case with averages over balls; see \cite{Kov19}.
This leads us to an interesting new open problem.

\begin{problem}
Prove or disprove that the assumptions of Theorem~\ref{thm:meainrestr} imply \eqref{eq:restr_ellipsoid} for every $f\in\textup{L}^p(\mathbb{R}^d)$ and for $\sigma$-a.e.\@ $\xi\in S$.
\end{problem}

The following question related to property \eqref{eq:restr_Lebesgue} remained open after \cite{Kov19} and its particular cases have already been studied by Ramos \cite{Ram20,Ram22} and Fraccaroli \cite{Fra21}.

\begin{problem}
Prove or disprove that, for every $f\in\textup{L}^p(\mathbb{R}^d)$, the assumptions of Theorem~\ref{thm:meainrestr} imply that $\sigma$-almost every point $\xi\in S$ is the Lebesgue point of $\widehat{f}$, in the sense that
\[ \lim_{t\to0+} \frac{1}{|B_{t}|} \int_{B_{t}(\xi)} \big| \widehat{f}(\eta) - (\mathcal{R}f)(\xi) \big| \,\textup{d}\eta = 0. \]
\end{problem}

In other words, here we do not know how to extend the range $1\leq s\leq 2p/(p+1)$, even when we only consider balls instead of arbitrary ellipsoids. 

The general maximal principle from \cite{Kov19}, concluding something about the Lebesgue sets of Fourier transforms $\widehat{f}$ from restriction estimates \eqref{eq:restr_apriori}, has been used by Bilz \cite{Bil20}.
It would be interesting to find similar applications of the stronger property \eqref{eq:restr_Lebesgue}.

\smallskip
The main new ingredient in the proof of Theorem~\ref{thm:meainrestr} will be a multi-parameter variant of the Christ--Kiselev lemma \cite{CK01}.
Even if its generalization is somewhat straightforward, we will argue that it is substantial by using it to deduce the following result on the Fourier transform alone, with no restriction phenomena involved.

\begin{theorem}
\label{thm:Carleson}
\begin{itemize}
\item[(a)]
For $p\in[1,2)$ and $f\in\textup{L}^p(\mathbb{R}^d)$ we have the maximal estimate
\begin{equation*}
    \Big\|\sup_{R_1,\ldots,R_d\in(0,\infty)}\big|\mathcal{F}\big(f\mathbbm{1}_{[-R_1,R_1]\times\cdots\times[-R_d,R_d]}\big) \big| \Big\|_{\textup{L}^{p'}(\mathbb{R}^d)} \lesssim_{d,p} \|f\|_{\textup{L}^p(\mathbb{R}^d)}
\end{equation*}
and $d$-parameter convergence
\begin{equation}\label{eq:Fourier_conv}
\lim_{R_1\to\infty,\ldots,R_d\to\infty} \int_{[-R_1,R_1]\times\cdots\times[-R_d,R_d]} f(x) e^{-2\pi\mathbbm{i}x\cdot\xi} \,\textup{d}x = \widehat{f}(\xi) 
\end{equation}
holds for a.e.\@ $\xi\in\mathbb{R}^d$.
\item[(b)]
If $d\ge 2$, then there exist a function $f\in \textup{L}^2(\R^d)$ and a set of positive measure $Q\subseteq \R^d$  such that
\begin{equation}
    \label{eq:Fourier_diver}
    \limsup_{R_1\to\infty,\ldots,R_d\to\infty} \bigg|\int_{[-R_1,R_1]\times\cdots\times[-R_d,R_d]} f(x) e^{-2\pi\mathbbm{i}x\cdot\xi} \,\textup{d}x\bigg|=\infty \quad \text{for every } x\in Q.
\end{equation}
In particular, even the weak $\textup{L}^2$ estimate
\begin{equation*}
    \Big\|\sup_{R_1,\ldots,R_d\in(0,\infty)}\big|\mathcal{F}\big(f\mathbbm{1}_{[-R_1,R_1]\times\cdots\times[-R_d,R_d]}\big) \big| \Big\|_{\textup{L}^{2,\infty}(\mathbb{R}^d)} \lesssim_{d} \|f\|_{\textup{L}^2(\mathbb{R}^d)}
\end{equation*}
\underline{does not hold}.
\end{itemize}
\end{theorem}

Part (a) can be thought of as a \emph{multi-parameter Menshov--Paley--Zygmund theorem}, while part (b) gives a counterexample to the corresponding multi-parameter analogue of Carleson's theorem \cite{Car66}.
The latter is not our original result, but a mere adaptation of the argument by Charles Fefferman \cite{Fef71} to the continuous setting. We include its detailed proof too for completeness of the exposition.

\smallskip
Finally, connections between the Fourier restriction problem and PDEs have been known since the work of Strichartz \cite{Str77}.
Let us comment on a certain reformulation of \eqref{eq:restr_multimax} in that direction.
The following standard setting is taken from the textbook by Tao \cite{Taobook}; also see the lecture notes by Koch, Tataru, and Vi\c{s}an \cite{KTVbook}.
Let $\phi\colon\mathbb{R}^n\to\mathbb{R}$ be a $\textup{C}^\infty$ function. 
A self-adjoint operator $\phi(D)=\phi(\nabla/2\pi\mathbbm{i})$ is defined to be the Fourier multiplier associated with the symbol $\phi$, i.e., 
\[ (\widehat{\phi(D) f})(\xi) = \phi(\xi) \widehat{f}(\xi). \]
If $\phi$ happens to be a polynomial
\[ \phi(\xi) = \sum_{|\alpha|\leq k} c_\alpha \xi^\alpha \]
in $n$ variables $\xi=(\xi_1,\ldots,\xi_n)$ of degree $k$ with real coefficients $c_\alpha$, then $\phi(D)$ is just the self-adjoint differential operator acting on Schwartz functions,
\[ \phi(D) = \sum_{|\alpha|\leq k} (2\pi\mathbbm{i})^{-|\alpha|} c_\alpha \partial^\alpha . \]
The solution of a general scalar constant-coefficient linear dispersive initial value problem
\begin{equation}\label{eq:dispPDE}
\begin{cases}
\partial_t u(x,t) = \mathbbm{i} \phi(D) u(x,t) & \text{in } \mathbb{R}^n\times\mathbb{R}, \\
\ \ u(x,0) = f(x) & \text{in } \mathbb{R}^n
\end{cases}
\end{equation}
is given explicitly as
\[  u(x,t) = (e^{\mathbbm{i} t \phi(D)} f)(x) := \int_{\mathbb{R}^n} e^{\mathbbm{i} t \phi(\xi) + 2\pi\mathbbm{i} x\cdot\xi} \widehat{f}(\xi) \,\textup{d}\xi \]
for $x\in\mathbb{R}^n$, $t\in\mathbb{R}$, and a Schwartz function $f\in\mathcal{S}(\mathbb{R}^d)$; see \cite[Section~2.1]{Taobook}.

\begin{corollary}
\label{cor:Strichartz}
Suppose that a Strichartz-type estimate for \eqref{eq:dispPDE} of the form
\begin{equation}\label{eq:Strichartz}
\big\| (e^{\mathbbm{i} t \phi(D)} f)(x) \big\|_{\textup{L}^s_{(x,t)}(\mathbb{R}^n\times\mathbb{R})} \lesssim_{n,\phi} \|f\|_{\textup{L}^2(\mathbb{R}^n)}
\end{equation}
holds for some exponent $s\in(2,\infty)$ and every Schwartz function $f\in\mathcal{S}(\mathbb{R}^n)$.
Then for every $\psi\in\mathcal{S}(\mathbb{R}^{n+1})$ and any choice of measurable functions $r_1,\ldots,r_{n+1}\colon\mathbb{R}^n\to\mathbb{R}$ the pseudo-differential operator
\[ (T_{\psi,r_1,\ldots,r_{n+1}}f)(x,t) :=  \int_{\mathbb{R}^n} \psi\big(r_1(\xi)x_1,\ldots,r_n(\xi)x_n,r_{n+1}(\xi)t\big) \,e^{\mathbbm{i} t \phi(\xi) + 2\pi\mathbbm{i} x\cdot\xi} \widehat{f}(\xi) \,\textup{d}\xi \]
satisfies the same bound 
\begin{equation}\label{eq:maxStrichartz}
\| T_{\psi,r_1,\ldots,r_{n+1}} f \|_{\textup{L}^s(\mathbb{R}^{n+1})} \lesssim_{n,\phi,\psi,s} \|f\|_{\textup{L}^2(\mathbb{R}^n)},
\end{equation}
with a constant that is independent of $r_1,\ldots,r_{n+1}$.
\end{corollary}

Note that \eqref{eq:Strichartz} is a particular case of \eqref{eq:maxStrichartz}, as the former inequality can be easily recovered by taking $r_1,\ldots,r_{n+1}$ to be identically $0$.
Specifically for the Schr\"{o}dinger equation, i.e., when $\phi(D)=\Delta$, the Strichartz estimate \eqref{eq:Strichartz} holds with $s=2+4/n$. A larger range of Strichartz estimates is available when one introduces the mixed norms \cite{BP61}, see \cite[Theorem~2.3]{Taobook} or a review paper \cite{CZ2010}, but our proof of Corollary~\ref{cor:Strichartz} is not well suited for this generalization.

While \eqref{eq:maxStrichartz} is questionably interesting in the theory of PDEs, we merely wanted to present a restatement of \eqref{eq:restr_multimax} in that language.
Note that in the definition of the above pseudo-differential operator it is only physically meaningful to scale the spatial variable $x$ and the time variable $t$ independently. In other words, just writing $\psi(r(\xi)(x,t))$ would make no sense.
This also partly motivates the study of multiparameter maximal Fourier restriction estimates.


\section{Multi-parameter Christ--Kiselev lemma}
This section is devoted to a bound on rather general multi-parameter maximal operators, which generalizes a classical result of Christ and Kiselev \cite{CK01}.

Let $(\mathbb{X},\mathcal{X},\mu)$ and $(\mathbb{Y},\mathcal{Y},\nu)$ be measure spaces. 
Let $d$ be a positive integer, which we interpret as the number of ``parameters.''
For every $1\leq j\leq d$ we are also given a countable totally ordered set $I_j$ and an increasing system $(E_j(i) : i\in I_j)$ of sets from $\mathcal{Y}$, i.e., an increasing function $E_j\colon I_j\to\mathcal{Y}$ with respect to the order on $I_j$ and the set inclusion on $\mathcal{Y}$.

\begin{lemma}[Multi-parameter Christ--Kiselev lemma]
\label{lm:multicklemma}
Take exponents $1\leq p<q\leq\infty$ and a bounded linear operator $T\colon\textup{L}^p(\mathbb{Y},\mathcal{Y},\nu)\to\textup{L}^q(\mathbb{X},\mathcal{X},\mu)$.
The maximal operator 
\[ (T_{\star}f)(x) := \sup_{(i_1,\ldots,i_d)\in I_1\times\cdots\times I_d} \big|T\big(f\mathbbm{1}_{E_1(i_1)\cap\cdots\cap E_d(i_d)}\big)(x)\big| \]
is also bounded from $\textup{L}^p(\mathbb{Y},\mathcal{Y},\nu)$ to $\textup{L}^q(\mathbb{X},\mathcal{X},\mu)$ with the operator norm satisfying
\begin{equation}\label{eq:ckest}
\|T_{\star}\|_{\textup{L}^p(\mathbb{Y})\to\textup{L}^q(\mathbb{X})} \leq \big(1-2^{1/q-1/p}\big)^{-d} \|T\|_{\textup{L}^p(\mathbb{Y})\to\textup{L}^q(\mathbb{X})}.
\end{equation}
\end{lemma}

The particular case $d=1$ is precisely \cite[Theorem~1.1]{CK01}.
The proof given below is a $d$-parameter modification of the approach from \cite{CK01}, incorporating a simplification due to Tao \cite[Note~\#2]{Tao06}, who used an induction on the cardinality of $I_1$ to immediately handle general measure spaces with atoms.
We include all details, since we desire to have a self-contained exposition.

\begin{proof}
By the monotone convergence theorem it is sufficient to prove the claim when the ordered sets $I_1,\ldots,I_d$ are finite. Note that it is crucial that the desired bound does not depend anyhow on their sizes. Thus, the proof will only consider finite index sets $I_j$.
The exponents $p$ and $q$, the two measure spaces, and the operator $T$ are fixed throughout the proof.
We are using a nested mathematical induction, first on $d$ and then on the cardinality of $I_d$, to prove \eqref{eq:ckest} for all finite increasing systems of sets $(E_j(i):i\in I_j)$, $1\leq j\leq d$. 
The induction basis $d=1=|I_1|$ is trivial, since then $T_{\star}$ satisfies the same bound as $T$.

We turn to the induction step.
By renaming the indices we can achieve that $I_j=\{1,2,\ldots,n_j\}$ for each $1\leq j\leq d$ and some positive integers $n_1,\ldots,n_d$.
Denote
\[ F(i) := E_1(n_1)\cap\cdots\cap E_{d-1}(n_{d-1})\cap E_d(i) \quad\text{for } 1\leq i\leq n_d. \]
Take a function $f\in\textup{L}^p(\mathbb{Y},\mathcal{Y},\nu)$.
By the assumption that the system $(E_d(i):i\in I_d)$ is increasing, we have
\[ 0 \leq \|f\|_{\textup{L}^p(F(1))} \leq \|f\|_{\textup{L}^p(F(2))} \leq\cdots\leq \|f\|_{\textup{L}^p(F(n_d))}. \]
Let $1\leq l\leq n_d$ be the smallest integer such that
\[ \|f\|_{\textup{L}^p(F(l))}^p \geq \frac{1}{2} \|f\|_{\textup{L}^p(F(n_d))}^p. \]
If $l\geq2$, then
\[ \|f\|_{\textup{L}^p(F(l-1))}^p 
< \frac{1}{2} \|f\|_{\textup{L}^p(F(n_d))}^p
\leq \frac{1}{2} \|f\|_{\textup{L}^p(\mathbb{Y})}^p, \]
so applying the induction hypothesis with the last system of sets replaced with the subsystem
\[ (E_{d}(i_d) : i_d\in\{1,\ldots,l-1\}), \]
we get
{\allowdisplaybreaks
\begin{align}
& \Big\| \max_{\substack{i_1,\ldots,i_{d}\\1\leq i_d\leq l-1}} \big|T\big(f\mathbbm{1}_{E_1(i_1)\cap\cdots\cap E_d(i_d)}\big)\big| \Big\|_{\textup{L}^q(\mathbb{X})} \nonumber \\
& = \Big\| \max_{\substack{i_1,\ldots,i_{d}\\1\leq i_d\leq l-1}} \big|T\big(f\mathbbm{1}_{F(l-1)}\mathbbm{1}_{E_1(i_1)\cap\cdots\cap E_d(i_d)}\big)\big| \Big\|_{\textup{L}^q(\mathbb{X})} \nonumber \\
& \leq \big(1-2^{1/q-1/p}\big)^{-d} \|T\|_{\textup{L}^p(\mathbb{Y})\to\textup{L}^q(\mathbb{X})} \|f\mathbbm{1}_{F(l-1)}\|_{\textup{L}^p(\mathbb{Y})} \nonumber \\
& \leq 2^{-1/p} \big(1-2^{1/q-1/p}\big)^{-d} \|T\|_{\textup{L}^p(\mathbb{Y})\to\textup{L}^q(\mathbb{X})} \|f\|_{\textup{L}^p(\mathbb{Y})}. \label{eq:inducthyp1}
\end{align}
}
Also,
\begin{align*}
\|f\|_{\textup{L}^p(F(n_d)\setminus F(l))}^p 
& = \|f\|_{\textup{L}^p(F(n_d))}^p - \|f\|_{\textup{L}^p(F(l))}^p
\leq \frac{1}{2} \|f\|_{\textup{L}^p(F(n_d))}^p
\leq \frac{1}{2} \|f\|_{\textup{L}^p(\mathbb{Y})}^p,
\end{align*}
so, if $l\leq n_d-1$, then applying the induction hypothesis with the last system of sets replaced with the subsystem,
\[ (E_{d}(i_d) : i_d\in\{l+1,\ldots,n_d\}), \]
we obtain
{\allowdisplaybreaks
\begin{align}
& \Big\| \max_{\substack{i_1,\ldots,i_{d}\\l+1\leq i_d\leq n_d}} \big|T\big(f\mathbbm{1}_{E_1(i_1)\cap\cdots\cap E_{d-1}(i_{d-1})\cap (E_d(i_d)\setminus E_d(l))}\big)\big| \Big\|_{\textup{L}^q(\mathbb{X})} \nonumber \\
& = \Big\| \max_{\substack{i_1,\ldots,i_{d}\\l+1\leq i_d\leq n_d}} \big|T\big(f\mathbbm{1}_{F(n_d)\setminus F(l)}\mathbbm{1}_{E_1(i_1)\cap\cdots\cap E_{d-1}(i_{d-1})\cap (E_d(i_d)\setminus E_d(l))}\big)\big| \Big\|_{\textup{L}^q(\mathbb{X})} \nonumber \\
& \leq \big(1-2^{1/q-1/p}\big)^{-d} \|T\|_{\textup{L}^p(\mathbb{Y})\to\textup{L}^q(\mathbb{X})} \|f\mathbbm{1}_{F(n_d)\setminus F(l)}\|_{\textup{L}^p(\mathbb{Y})} \nonumber \\
& \leq 2^{-1/p} \big(1-2^{1/q-1/p}\big)^{-d} \|T\|_{\textup{L}^p(\mathbb{Y})\to\textup{L}^q(\mathbb{X})} \|f\|_{\textup{L}^p(\mathbb{Y})}. \label{eq:inducthyp2}
\end{align}
}
Finally, if $d\geq2$, then we can also apply the induction hypothesis with the same first $d-1$ systems of sets, to conclude
{\allowdisplaybreaks
\begin{align}
& \Big\| \max_{i_1,\ldots,i_{d-1}} \big|T\big(f\mathbbm{1}_{E_1(i_1)\cap\cdots\cap E_{d-1}(i_{d-1})\cap E_d(l)}\big)\big| \Big\|_{\textup{L}^q(\mathbb{X})} \nonumber \\
& =\Big\| \max_{i_1,\ldots,i_{d-1}} \big|T\big(f\mathbbm{1}_{E_d(l)}\mathbbm{1}_{E_1(i_1)\cap\cdots\cap E_{d-1}(i_{d-1})}\big)\big| \Big\|_{\textup{L}^q(\mathbb{X})} \nonumber \\
& \leq \big(1-2^{1/q-1/p}\big)^{-d+1} \|T\|_{\textup{L}^p(\mathbb{Y})\to\textup{L}^q(\mathbb{X})} \|f\mathbbm{1}_{E_d(l)}\|_{\textup{L}^p(\mathbb{Y})} \nonumber \\
& \leq \big(1-2^{1/q-1/p}\big)^{-d+1} \|T\|_{\textup{L}^p(\mathbb{Y})\to\textup{L}^q(\mathbb{X})} \|f\|_{\textup{L}^p(\mathbb{Y})}. \label{eq:inducthyp3}
\end{align}
}
The last bound also holds in the case $d=1$, with the maximum disappearing from the left hand side, and it is a consequence of the mere boundedness of $T$.

Now denote
\begin{align*}
S := \big\{x\in\mathbb{X} :\, & (T_{\star}f)(x)=\big|\big(T\big(f\mathbbm{1}_{E_1(i_1)\cap\cdots\cap E_d(i_d)}\big)\big)(x)\big| \\ & \text{for some } (i_1,\ldots,i_d)\in I_1\times\cdots\times I_d \text{ such that } i_d\leq l-1 \big\},
\end{align*}
so, using linearity of $T$,
\begin{align*}
T_{\star}f \leq\, & \mathbbm{1}_S \max_{\substack{i_1,\ldots,i_d\\1\leq i_d\leq l-1}} \big|T\big(f\mathbbm{1}_{E_1(i_1)\cap\cdots\cap E_{d-1}(i_{d-1})\cap E_d(i_d)}\big)\big| \\
& + \mathbbm{1}_{\mathbbm{X}\setminus S} \max_{\substack{i_1,\ldots,i_d\\l+1\leq i_d\leq n_d}} \big|T\big(f\mathbbm{1}_{E_1(i_1)\cap\cdots\cap E_{d-1}(i_{d-1})\cap (E_d(i_d)\setminus E_d(l))}\big)\big| \\
& + \mathbbm{1}_{\mathbbm{X}\setminus S} \max_{i_1,\ldots,i_{d-1}} \big|T\big(f\mathbbm{1}_{E_1(i_1)\cap\cdots\cap E_{d-1}(i_{d-1})\cap E_d(l)}\big)\big|.
\end{align*}
Here, maximum over an empty set is understood to be $0$.
When $q<\infty$ we conclude
\begin{align*}
\|T_{\star}f\|_{\textup{L}^q(\mathbb{X})}
& \leq \bigg(\Big\|\max_{\substack{i_1,\ldots,i_d\\1\leq i_d\leq l-1}} \big|T\big(f\mathbbm{1}_{E_1(i_1)\cap\cdots\cap  E_d(i_d)}\big)\big|\Big\|_{\textup{L}^q(S)}^q \\
& \qquad + \Big\|\max_{\substack{i_1,\ldots,i_d\\l+1\leq i_d\leq n_d}} \big|T\big(f\mathbbm{1}_{E_1(i_1)\cap\cdots\cap E_{d-1}(i_{d-1})\cap (E_d(i_d)\setminus E_d(l))}\big)\big|\Big\|_{\textup{L}^q(\mathbb{X}\setminus S)}^q\bigg)^{1/q} \\
& \quad + \Big\|\max_{i_1,\ldots,i_{d-1}} \big|T\big(f\mathbbm{1}_{E_1(i_1)\cap\cdots\cap E_{d-1}(i_{d-1})\cap E_d(l)}\big)\big|\Big\|_{\textup{L}^q(\mathbb{X}\setminus S)},
\end{align*}
while in the endpoint case $q=\infty$ we instead have
\begin{align*}
\|T_{\star}f\|_{\textup{L}^\infty(\mathbb{X})}
& \leq \max\bigg\{\Big\|\max_{\substack{i_1,\ldots,i_d\\1\leq i_d\leq l-1}} \big|T\big(f\mathbbm{1}_{E_1(i_1)\cap\cdots\cap E_d(i_d)}\big)\big|\Big\|_{\textup{L}^\infty(S)}, \\
& \qquad\qquad \Big\|\max_{\substack{i_1,\ldots,i_d\\l+1\leq i_d\leq n_d}} \big|T\big(f\mathbbm{1}_{E_1(i_1)\cap\cdots\cap E_{d-1}(i_{d-1})\cap (E_d(i_d)\setminus E_d(l))}\big)\big|\Big\|_{\textup{L}^\infty(\mathbb{X}\setminus S)}\bigg\} \\
& \quad + \Big\|\max_{i_1,\ldots,i_{d-1}} \big|T\big(f\mathbbm{1}_{E_1(i_1)\cap\cdots\cap E_{d-1}(i_{d-1})\cap E_d(l)}\big)\big|\Big\|_{\textup{L}^\infty(\mathbb{X}\setminus S)}.
\end{align*}
Applying \eqref{eq:inducthyp1}, \eqref{eq:inducthyp2}, and \eqref{eq:inducthyp3} we complete the induction step.
\end{proof}

\begin{remark}
An alternative proof of Lemma~\ref{lm:multicklemma} can be obtained as follows. We can generalize the claim further to general sublinear operators $T$, i.e., operators satisfying
\[ |T(\alpha f)| = |\alpha| |Tf|, \quad |T(f+g)| \leq |Tf| + |Tg| \]
for all $\alpha\in\mathbb{C}$ and all $f,g\in\textup{L}^p(\mathbb{Y},\mathcal{Y},\nu)$.
The advantage of doing this is that various maximal operators are always sublinear.
Then we can write the operator $T_\star$ as a composition of $d$ maximal truncations, each one with respect to a single increasing system $(E_j(i):i\in I_j)$, namely
\[ T_{\star}f = \sup_{i_1 \in I_1} \sup_{i_2 \in I_2} \cdots \sup_{i_d \in I_d} \Big|T\Big(\cdots\big((f\mathbbm{1}_{E_1(i_1)})\mathbbm{1}_{E_2(i_2)}\big)\cdots\mathbbm{1}_{E_d(i_d)}\Big)\Big|, \]
so the claim is reduced merely to the one-parameter case.
Finally, one can notice that the known proofs of the particular case $d=1$, both the one by Christ and Kiselev \cite[Theorem~1.1]{CK01} and the one by Tao \cite[Note~\#2]{Tao06}, clearly remain valid for merely sublinear operators $T$.
We leave the details to the reader.
\end{remark}

Now assume that the second measurable space splits as a product
\[ (\mathbb{Y},\mathcal{Y}) = (\mathbb{Y}_1\times\cdots\times\mathbb{Y}_d,\, \mathcal{Y}_1\otimes\cdots\otimes\mathcal{Y}_d) \]
of $d\geq1$ measurable spaces $(\mathbb{Y}_j,\mathcal{Y}_j)$. 
Also suppose that for each $1\leq j\leq d$ we have a countable totally ordered set $I_j$ and an increasing system $(A_{j}^{i}:i\in I_j)$ of sets from $\mathcal{Y}_j$.

\begin{corollary}
\label{cor:multicklemma}
Take exponents $1\leq p<q\leq\infty$ and a bounded linear operator $T\colon\textup{L}^p(\mathbb{Y},\mathcal{Y},\nu)\to\textup{L}^q(\mathbb{X},\mathcal{X},\mu)$.
The maximal operator 
\begin{equation}\label{eq:maximal_op}
(T_{\star}f)(x) := \sup_{(i_1,\ldots,i_d)\in I_1\times\cdots\times I_d} \Big|T\Big(f\mathbbm{1}_{A^{i_1}_{1}\times\cdots\times A^{i_d}_{d}}\Big)(x)\Big|
\end{equation}
is also bounded from $\textup{L}^p(\mathbb{Y},\mathcal{Y},\nu)$ to $\textup{L}^q(\mathbb{X},\mathcal{X},\mu)$ with the operator norm satisfying
\[ \|T_{\star}\|_{\textup{L}^p(\mathbb{Y})\to\textup{L}^q(\mathbb{X})} \leq \big(1-2^{1/q-1/p}\big)^{-d} \|T\|_{\textup{L}^p(\mathbb{Y})\to\textup{L}^q(\mathbb{X})}. \]
\end{corollary}

\begin{proof}
This result is an immediate consequence of Lemma~\ref{lm:multicklemma}, obtained by taking
\[ E_j(i) = \mathbb{Y}_1 \times\cdots\times \mathbb{Y}_{j-1}\times A^{i}_{j} \times\mathbb{Y}_{j+1}\times\cdots\times\mathbb{Y}_d. \qedhere \]
\end{proof}

The constants blow up as $q$ approaches $p$.
An easy modification of the proof of Lemma~\ref{lm:multicklemma} gives the following endpoint result with logarithmic losses when the sets $I_j$ are finite.

\begin{corollary}
\label{cor:multirm}
Take an exponent $p\in[1,\infty]$ and a bounded linear operator $T\colon\textup{L}^p(\mathbb{Y},\mathcal{Y},\nu)\to\textup{L}^p(\mathbb{X},\mathcal{X},\mu)$.
The maximal operator given by \eqref{eq:maximal_op} satisfies
\[ \|T_{\star}\|_{\textup{L}^p(\mathbb{Y})\to\textup{L}^p(\mathbb{X})} 
\leq (\lceil\log_2 |I_1|\rceil+1)\cdots(\lceil\log_2 |I_d|\rceil+1) \,\|T\|_{\textup{L}^p(\mathbb{Y})\to\textup{L}^p(\mathbb{X})}. \]
\end{corollary}

Formulation of Corollary~\ref{cor:multirm} is motivated by Tao's \cite[Note~\#2, Q14]{Tao06}.
The particular case when $p=2$ and $T$ is the Fourier transform could be called the \emph{multi-parameter Rademacher--Menshov theorem}.
We will not need Corollary~\ref{cor:multirm} in the later text and we formulated it only for comparison with a very different method by Krause, Mirek, and Trojan \cite[Section~3]{KMT18}.


\section{Proof of Theorem~\ref{thm:meainrestr}}
Denote the maximal operator
\[ \mathcal{M}f := \sup_{r_1,\ldots,r_d\in(0,\infty)}\big|\widehat{f}\ast\chi_{r_1,\ldots,r_d}\big|. \]
We begin with an observation that $\widehat{f}\ast\chi_{r_1,\ldots,r_d}$ is the Fourier transform of
\[ (x_1,\ldots,x_d)\mapsto f(x_1,\ldots,x_d)\,\widecheck{\chi}(r_1 x_1,\ldots,r_d x_d). \]
Using \eqref{eq:chi_cond} and the fundamental theorem of calculus we expand, for any $(x_1,\ldots,x_d)\in(\mathbb{R}\setminus\{0\})^d$ and $(r_1,\ldots,r_d)\in(0,\infty)^d$,
\begin{align*}
& \widecheck{\chi}(r_1 x_1,\ldots,r_d x_d) \\
& = \sum_{\epsilon\in\{-1,1\}^d} (-1)^{\#\epsilon} \,\mathbbm{1}_{Q(\epsilon)}(x_1,\ldots,x_d) \!\!\!\!\!\!\!\!\!\!\!\!\!\!\! \int_{\{(t_1,\ldots,t_d)\in Q(\epsilon):|t_j|\geq r_j|x_j|\text{ for }1\leq j\leq d\}}\limits \!\!\!\!\!\!\!\!\!\!\!\!\!\!\!\!\!\!\!\! \big(\partial_1\cdots\partial_d\widecheck{\chi}\big)(t_1,\ldots,t_d) \,\textup{d}t_1 \cdots \textup{d}t_d \\
& = \sum_{\epsilon\in\{-1,1\}^d} (-1)^{\#\epsilon} \int_{Q(\epsilon)} \mathbbm{1}_{R(\epsilon;|t_1|/r_1,\ldots,|t_d|/r_d)}(x_1,\ldots,x_d) \,\big(\partial_1\cdots\partial_d\widecheck{\chi}\big)(t_1,\ldots,t_d) \,\textup{d}t_1 \cdots \textup{d}t_d.
\end{align*}
Here $Q(\epsilon)$ is the open coordinate ``quadrant'' determined by $\epsilon=(\epsilon_1,\ldots,\epsilon_d)\in\{-1,1\}^d$, i.e.,
\[ Q(\epsilon) := \{(x_1,\ldots,x_d)\in\mathbb{R}^d \,:\, \mathop{\textup{sgn}}x_j=\epsilon_j \text{ for } 1\leq j\leq d \}, \]
$\#\epsilon$ denotes the number of $1$'s among the coordinates of $\epsilon$, and we also denote
\begin{equation}\label{eq:rectangles}
R(\epsilon;s_1,\ldots,s_d) := Q(\epsilon)\cap([-s_1,s_1]\times\cdots\times[-s_d,s_d])
\end{equation}
for any $s_1,\ldots,s_d\in(0,\infty)$.
Multiplying by $f$ and taking Fourier transforms we obtain the pointwise identity
\[ \widehat{f}\ast\chi_{r_1,\ldots,r_d} = \sum_{\epsilon\in\{-1,1\}^d} (-1)^{\#\epsilon} \int_{Q(\epsilon)} \mathcal{F}\big(f\,\mathbbm{1}_{R(\epsilon;|t_1|/r_1,\ldots,|t_d|/r_d)}\big) \,\big(\partial_1\cdots\partial_d\widecheck{\chi}\big)(t_1,\ldots,t_d) \,\textup{d}t_1 \cdots \textup{d}t_d, \]
so that
\[ \mathcal{M}f \leq \int_{\mathbb{R}^d} \Big( \sup_{r_1,\ldots,r_d\in(0,\infty)} \big|\mathcal{F}\big(f\,\mathbbm{1}_{R(\epsilon;|t_1|/r_1,\ldots,|t_d|/r_d)}\big) \big| \Big) \,\big|\big(\partial_1\cdots\partial_d\widecheck{\chi}\big)(t_1,\ldots,t_d)\big| \,\textup{d}t_1 \cdots \textup{d}t_d. \]
Note that each of the sets \eqref{eq:rectangles} is a $d$-dimensional rectangle in $\mathbb{R}^d$, so invoking Corollary~\ref{cor:multicklemma} with $T=\mathcal{F}$, which is known to satisfy \eqref{eq:restr_apriori}, gives
\[ \Big\|\sup_{r_1,\ldots,r_d\in(0,\infty)\cap\mathbb{Q}}\big|\mathcal{F}\big(f\,\mathbbm{1}_{R(\epsilon;|t_1|/r_1,\ldots,|t_d|/r_d)}\big)\big|\Big\|_{\textup{L}^q(S,\sigma)} \lesssim_{d,\sigma,p,q} \|f\|_{\textup{L}^p(\mathbb{R}^d)}. \]
The last implicit constant is independent of $t_1,\ldots,t_d$, so integrability of $\partial_1\cdots\partial_d\widecheck{\chi}$, thanks to \eqref{eq:chi_cond} again, establishes
\begin{equation}\label{eq:proof_max}
\|\mathcal{M}f\|_{\textup{L}^q(S,\sigma)} \lesssim_{d,\sigma,\chi,p,q} \|f\|_{\textup{L}^p(\mathbb{R}^d)},
\end{equation}
which is precisely \eqref{eq:restr_multimax}.

The proof of \eqref{eq:restr_convergence} is now standard. 
The claim is clear for $f\in\textup{L}^1(\mathbb{R}^d)$. By
\[ \textup{L}^s(\mathbb{R}^d) \subseteq \textup{L}^1(\mathbb{R}^d) + \textup{L}^p(\mathbb{R}^d) \]
it is sufficient to verify it when $f\in\textup{L}^p(\mathbb{R}^d)$. For any $\varepsilon>0$ denote the exceptional set
\[ E_\varepsilon := \Big\{ \xi\in\mathbb{R}^d \,:\, \inf_{r\in(0,\infty)} \sup_{r_1,\ldots,r_d\in(0,r]} \big| \big(\widehat{f}\ast\chi_{r_1,\ldots,r_d}\big)(\xi) - (\mathcal{R}f)(\xi) \big| \geq\varepsilon \Big\}, \]
observing that \eqref{eq:restr_convergence} holds for every point outside of $\cup_{\varepsilon\in(0,\infty)}E_\varepsilon$.
It is easy to see that for every $g\in\mathcal{S}(\mathbb{R}^d)$ by the mere continuity of $\widehat{g}$ we have
\[ \lim_{(0,\infty)^d\ni(r_1,\ldots,r_d)\to(0,\ldots,0)} \widehat{g}\ast\chi_{r_1,\ldots,r_d} = \mathcal{R}g \]
pointwise on $S$ and, consequently,
\[ E_\varepsilon \subseteq \Big\{\xi\in S \,:\, \mathcal{M}(f-g)(\xi)\geq\frac{\varepsilon}{2} \Big\} \cup \Big\{\xi\in S \,:\, \mathcal{R}(f-g)(\xi)\geq\frac{\varepsilon}{2} \Big\}. \]
Thus, Estimates \eqref{eq:proof_max}, \eqref{eq:restr_apriori} and the Markov--Chebyshev inequality give
\[ \sigma(E_\varepsilon) \lesssim \varepsilon^{-q} \|f-g\|_{\textup{L}^p(\mathbb{R}^d)}^q. \]
By the density of $\mathcal{S}(\mathbb{R}^d)$ in $\textup{L}^p(\mathbb{R}^d)$ we conclude $\sigma(E_\varepsilon)=0$ and nestedness of these sets also gives $\sigma(\cup_{\varepsilon\in(0,\infty)}E_\varepsilon)=0$.
Thus, \eqref{eq:restr_convergence} really holds for $\sigma$-almost every $\xi\in S$.

Turning to \eqref{eq:restr_Lebesgue}, we define
\[ \big(\widetilde{\mathcal{M}}f\big)(\xi) := \sup_{r_1,\ldots,r_d\in(0,\infty)} \frac{1}{|B_{r_1,\ldots,r_d}|} \int_{B_{r_1,\ldots,r_d}(\xi)} \big| \widehat{f}(\eta) \big| \,\textup{d}\eta \]
and repeat a trick from \cite{MRW19}. It is again sufficient to verify the claim in the endpoint case $f\in\textup{L}^{2p/(p+1)}(\mathbb{R}^d)$. Define
\[ g(x) := \int_{\mathbb{R}^d} f(y) \overline{f(y-x)} \,\textup{d}y, \]
so that $g\in\textup{L}^p(\mathbb{R}^d)$ and $\widehat{g}(\xi) = \big|\widehat{f}(\xi)\big|^2$.
Choose any non-negative $\chi\in\mathcal{S}(\mathbb{R}^d)$ with integral $1$ that is strictly positive on the closed unit ball $B_1(0,\ldots,0)$. Then, by the Cauchy--Schwartz inequality,
\begin{align*}
\frac{1}{|B_{r_1,\ldots,r_d}|} \int_{B_{r_1,\ldots,r_d}(\xi)} \big| \widehat{f}(\eta) \big| \,\textup{d}\eta
& \leq \bigg( \frac{1}{|B_{r_1,\ldots,r_d}|} \int_{B_{r_1,\ldots,r_d}(\xi)} \big| \widehat{f}(\eta) \big|^2 \,\textup{d}\eta \bigg)^{1/2} \\
& \lesssim_{\chi} \big(\widehat{g}\ast\chi_{r_1,\ldots,r_d}\big)(\xi)^{1/2},
\end{align*}
so the bound \eqref{eq:proof_max} applied to $g$ gives
\[ \big\|\widetilde{\mathcal{M}}f\big\|_{\textup{L}^{2q}(S,\sigma)}
\lesssim_{\chi} \|\mathcal{M}g\|_{\textup{L}^q(S,\sigma)}^{1/2}
\lesssim_{d,\sigma,\chi,p,q} \|g\|_{\textup{L}^p(\mathbb{R}^d)}^{1/2} 
\leq \|f\|_{\textup{L}^{2p/(p+1)}(\mathbb{R}^d)}. \]
Now we can repeat exactly the same density argument as before to conclude that \eqref{eq:restr_Lebesgue} holds for $\sigma$-almost every $\xi\in S$.
Finally, \eqref{eq:restr_ellipsoid} is an obvious consequence of \eqref{eq:restr_Lebesgue} and the triangle inequality.


\section{Proof of Theorem~\ref{thm:Carleson}}
The maximal operator appearing in the part (a) is simply $T_\star$ from \eqref{eq:maximal_op}, where $\mathbb{X}=\mathbb{R}^d$, $\mathbb{Y}_j=\mathbb{R}$, $T=\mathcal{F}$, $q=p'$, $I_j=(0,\infty)\cap\mathbb{Q}$, and $A_j^R=[-R,R]$.
Note that we use $p<2$ in the condition $p<p'=q$, so that Corollary~\ref{cor:multicklemma} applies and deduces the desired estimate from the well-known fact that the Fourier transform $\mathcal{F}$ is bounded from $\textup{L}^p(\mathbb{R}^d)$ to $\textup{L}^{p'}(\mathbb{R}^d)$.
The convergence result is then proved via exactly the same density argument as the one used in the previous section.

We turn to the part (b). It will be merely an adaptation of Fefferman's argument \cite{Fef71} to the continuous case. We present the complete proof here because the construction was only outlined in the aforementioned paper and it is necessary for us to construct the function in $\textup{L}^2(\R^d)$ for which the limit \eqref{eq:Fourier_conv} does not exist, instead of just disproving $\textup{L}^2(\R^d)\to \textup{L}^{2,\infty}(\R^d) $ boundedness. Namely, Stein's maximal principle \cite{Ste61} does not apply in the case of non-compact groups, such as $\R^d$.

We define $D_R(t):=\sin(2\pi Rt)/\pi t$. The operator $S_{R_1,\dots, R_d}$ is defined on $\textup{L}^2(\R^d)$ as
\[S_{R_1,\dots, R_d}f: = \mathcal{F}(\widecheck{f}\1_{[-R_1,R_1]\times\dots \times [-R_d,R_d]}) = f\ast (D_{R_1}\otimes\dots \otimes D_{R_d}).\]
Here $u_1\otimes\cdots\otimes u_d$ denotes the elementary tensor made of one-dimensional functions, defined as
\[ (u_1\otimes\cdots\otimes u_d)(x_1,\ldots,x_d) := u_1(x_1) \cdots u_d(x_d). \]
Observe that Young's convolution inequality implies 
\[ \|{S_{R_1,\dots, R_d}}\|_{\textup{L}^2(\R^d)\to \textup{L}^{\infty}(\R^d)}\lesssim (R_1\cdots R_d)^{1/2}. \]

Following Fefferman's example, we use the following definition throughout the remainder of this section. For $\lambda\in \R$ we define 
\[f_{\lambda}(x_1,x_2):=e^{2\pi \mathbbm{i} \lambda  x_1x_2}\1_{[-2,2]^2}(x_1,x_2).\]
The next lemma gives bounds that are crucial  for the proof.

\begin{lemma}\label{lm:cntexmpl_main}
\
\begin{itemize}
    \item[(a)] There exists $C>0$ such that for all $x_1,x_2\in [2/3,1]$ the following holds: \[|{S_{\lambda x_2, \lambda x_1} f_{\lambda}(x_1,x_2)}| \ge C\log \lambda\] 
    whenever $\lambda$ is large enough.
    \item[(b)] There exists $C>0$ such that for all $x_1,x_2\in [2/3,1]$ and $\lambda'\ge 3 \lambda>0$ the following holds: 
    \[|{S_{\lambda' x_2, \lambda' x_1} f_\lambda (x_1,x_2)}| \leq C.\]
\end{itemize}
\end{lemma}

Before proving the lemma, we prove that it implies the part (b) of Theorem \ref{thm:Carleson}. 
We will prove that there exist a function $f\in \textup{L}^2(\R^d)$ and a number $\delta >0$ such that 
\begin{equation}
\label{eq:Fourier_explicit_lower}
  \limsup_{R_1\to\infty,\dots, R_d\to\infty}|{S_{R_1,\dots, R_d}f(x_1,\dots, x_d)}| =\infty \quad \text{for every } (x_1,\dots, x_d)\in \left[ 2/3, 1\right]^2 \times [-\delta, \delta ]^{d-2},  
\end{equation}
so the function $\widecheck{f}\in \textup{L}^2(\R^d)$ will be the one for which \eqref{eq:Fourier_diver} holds.

Let $\psi\in \mathcal{S}(\R)$ be a real-valued Schwartz function such that $\psi(0)>0$ and $\operatorname{supp}(\widecheck{\psi})\subseteq [-1,1]$. For the function $F(x_1,\dots, x_d):= f(x_1,x_2)\prod_{j=3}^{d}\psi(x_j)$, because of the assumption on the support of $\widecheck{\psi}$, we have
\[\limsup_{R_1\to\infty,\dots, R_d\to\infty}|{S_{R_1,\dots, R_d} F(x_1,\dots, x_d)}| = \limsup_{R_1\to\infty, R_2\to\infty}\Big|{S_{R_1,R_2}f(x_1,x_2) \prod_{j=3}^{d}\psi(x_j)}\Big|.\]
Furthermore, since $\psi(0)>0$, there exists some $\delta >0$ such that $\psi(x)>0$ for all $x\in [-\delta, \delta]$, so it is enough to prove \eqref{eq:Fourier_explicit_lower} for $d=2$.

We define the sequence of positive real numbers $(a_k)_{k=1}^{\infty}$ recursively as $a_1=1$, $a_{k+1}=2^{-k/a_k}$ and the sequence of positive real numbers $(\lambda_k)_{k=1}^{\infty}$ with $\lambda_k=a_{k+1}^{-1}$. Observing that $\sum_{k=1}^{\infty}a_k< \infty$, it follows that the function
\[f(x_1,x_2):= \sum_{k=0}^{\infty}a_k f_{\lambda_k}(x_1,x_2)\]
is well defined and in $\textup{L}^2(\R^2)$

We claim that there exist real numbers $C_i>0$, $i=1,2,3$ such that the following inequalities hold for all $x_1,x_2\in [2/3,1]$ and $n\in \N$:
\begin{itemize}
    \item[(1)] $|{S_{\lambda_n x_2, \lambda_n x_1}f_{\lambda_n}(x_1,x_2)}|\ge C_1 \log\lambda_n$,
    \item[(2)] $|{S_{\lambda_n x_2, \lambda_n x_1}f_{\lambda_k}(x_1,x_2)}|\leq C_2$ when $k<n$,
    \item[(3)] $|{S_{\lambda_n x_2, \lambda_n x_1}f_{\lambda_k}(x_1,x_2)}|\leq C_3\lambda_n$ when $k>n$.
\end{itemize}
Indeed, since $\lambda_{k+1}\ge 4\lambda_k$ for all $k\in\N$, the first two inequalities follow from Lemma~\ref{lm:cntexmpl_main}, while the third one follows from Young's convolution inequality.
Therefore, observing that sequences satisfy $\lambda_{n}\sum_{k>n}a_k\lesssim 1$ for all $n\in \N$ and $a_n\log \lambda_n  \sim n$, for $x_1,x_2\in [2/3,1]$ and $n$ large enough it follows that
\begin{equation*}
    \begin{split}
        |{S_{\lambda_n x_2, \lambda_n x_1}f(x_1,x_2)}| 
        &\ge a_n|{S_{\lambda_n x_2, \lambda_n x_1}f_{\lambda_n}(x_1,x_2)}| - \sum_{k\neq n} a_k|{S_{\lambda_n x_2, \lambda_n x_1}f_{\lambda_k}(x_1,x_2)}|\\
        &\ge C_1 a_n\log \lambda_n - C_2\sum_{k<n}a_k - C_3\lambda_n\sum_{k>n}a_k \gtrsim n.
    \end{split}
\end{equation*}
Finally, noting that $\lambda_nx_2,\lambda_nx_1\to\infty$ as $n\to\infty$ finishes the proof of \eqref{eq:Fourier_explicit_lower} in the case $d=2$ and therefore also the part (b) of the theorem.

The following technical lemma will be needed in the proof of Lemma \ref{lm:cntexmpl_main}.
\begin{lemma}\label{lm:cntexmpl_crucial}
\
\begin{itemize}
    \item [(a)] There exist $C,\lambda_0>0$ such that 
    \[\bigg|{\operatorname{p.v.}\int_{-1}^{1}\int_{-1}^{1} \frac{e^{2\pi \mathbbm{i} \lambda x_1x_2}}{x_1x_2} \,\textup{d}x_1 \,\textup{d}x_2}\bigg| \ge C\log \lambda \quad \text{for every } \lambda\ge \lambda_0.\]
    \item [(b)] There exists $C>0$ such that for all $c_1,c_2\in \R$ for which $\max\{|{c_1}|,|{c_2}|\} \ge 4/3$, the following holds:
    \[\bigg|{\operatorname{p.v.}\int_{-1}^{1}\int_{-1}^{1} \frac{e^{2\pi \mathbbm{i}\lambda (x_1x_2 + c_1x_1+ c_2 x_2)}}{x_1x_2} \,\textup{d}x_1 \,\textup{d}x_2 }\bigg| \leq C \quad \text{for every } \lambda>0.\]
\end{itemize}
\end{lemma}

\begin{proof}
    (a) This was proved in \cite{MPTT04}, but we repeat the short proof for the completeness. Since $\int_{0}^{\infty}\sin t \,\textup{d}t /t= \pi/2$, there exists $\lambda_1>0$ such that 
    $\int_{0}^{x}\sin t \,\textup{d}t /t \in \left[\pi/4, 3\pi/4\right]$,  for all $x\ge \lambda_1$. Now, using symmetries of the integrand and change of variables, it follows
    \begin{equation*}
        \begin{split}
        \operatorname{p.v.}\int_{-1}^{1}\int_{-1}^{1} \frac{e^{2\pi \mathbbm{i}\lambda x_1x_2}}{x_1x_2} \,\textup{d}x_1\,\textup{d}x_2
        &= 4\mathbbm{i}\int_{0}^{1}\int_{0}^{1} \frac{\sin(2\pi \lambda x_1x_2)}{x_1x_2} \,\textup{d}x_1 \,\textup{d}x_2\\
        &= 4\mathbbm{i}\int_{0}^{2\pi}\frac{1}{x_2}\int_{0}^{\lambda x_2}\frac{\sin t}{t} \,\textup{d}t\,\textup{d}x_2 \\
        &= 4\mathbbm{i}\int_{0}^{{\lambda_1}/{\lambda}}\frac{1}{x_2} \int_{0}^{\lambda x_2}\frac{\sin t}{t} \,\textup{d}t \,\textup{d}x_2 + 4\mathbbm{i}\int_{{\lambda_1}/{\lambda}}^{2\pi} \int_{0}^{\lambda x_2}\frac{\sin t}{t} \,\textup{d}t \,\frac{\textup{d}x_2}{x_2}.
        \end{split}
    \end{equation*}
    
    For the first integral observe that $t\mapsto(\sin t)/t$ is absolutely bounded by $1$, so the integral is absolutely bounded by $\lambda_1$. For the second integral we use fact that $\lambda x_2\ge \lambda_1$ so:
    \[\int_{{\lambda_1}/{\lambda}}^{2\pi} \int_{0}^{\lambda x_2}\frac{\sin t}{t} \,\textup{d}t \,\frac{\textup{d}x_2}{x_2} \gtrsim\int_{{\lambda_1}/{\lambda}}^{2\pi} \frac{\textup{d}x_2}{x_2} = \log \lambda +\log(2\pi)  - \log \lambda_1.\]
    Finally, adding the two integrals and choosing $\lambda_0$ large enough compared to $\lambda_1$, the statement holds.
    
    (b) Assume, without loss of generality, that $c_1\ge 4/3$. Using symmetries of the integrand, it follows that
    \begin{equation*}
        \begin{split}
            &\operatorname{p.v.}\int_{-1}^{1}\int_{-1}^{1} \frac{e^{2\pi \lambda \mathbbm{i}(x_1x_2+c_1x_1+c_2x_2)}}{x_1x_2} \,\textup{d}x_1 \textup{d}x_2 \\
            &\quad \quad = \mathbbm{i}
            \operatorname{p.v.}\int_{-1}^{1}\int_{-1}^{1}  \frac{\sin(2\pi\lambda  x_1(x_2+c_1))}{x_1x_2} e^{2\pi \mathbbm{i}\lambda c_2 x_2}\,\textup{d}x_1\textup{d}x_2.
        \end{split}
    \end{equation*}
     If we define
    \[g_{\eps}(x_2):= 2\int_{\eps }^{1} \frac{\sin (2\pi \lambda x_1( x_2+c_1))}{x_1} \,\textup{d}x_1,\]
    from the assumption $c_1\ge 4/3$, it follows that $|{g_\eps'(x_2)}| \lesssim (x_2+c_1)^{-1}\lesssim 1$ for all $x_2\in [-1,1]$, where the implicit constant is independent of both $\lambda$ and $\eps$. Therefore,
    \begin{equation*}
        \begin{split}
            &\bigg|{\int_{([-1,1]\setminus[-\eps,\eps])^2} \frac{\sin(2\pi\lambda x_1(x_2+c_1))}{x_1} \frac{e^{2\pi \mathbbm{i}\lambda c_2x_2}}{x_2}\,\textup{d}x_1\textup{d}x_2}\bigg|
            = \bigg|{\int_{[-1,1]\setminus[-\eps,\eps]} g_{\eps}(x_2)\frac{e^{2\pi \mathbbm{i}\lambda c_2x_2}}{x_2}\,\textup{d}x_2}\bigg|\\ 
            &\quad \leq  \bigg|{\int_{[-1,1]\setminus[-\eps,\eps]} \frac{g_{\eps}(x_2)-g_{\eps}(0)}{x_2} e^{2\pi \mathbbm{i}\lambda c_2x_2}\,\textup{d}x_2}\bigg| + \bigg|{g_{\eps}(0)\int_{[-1,1]\setminus[-\eps,\eps]} \frac{e^{2\pi \mathbbm{i}\lambda c_2x_2}}{x_2}\,\textup{d}x_2}\bigg|\\
            &\quad\lesssim \int_{-1}^{1} \Big|{\sup_{t\in[-1,1]} g'(t)}\Big| \,\textup{d}x_2 + \sup_{N>0}\bigg|{\int_{0}^{N} \frac{\sin t}{t} dt }\bigg|^2 \lesssim 1.
        \end{split}
    \end{equation*}
    Letting $\eps\to 0$, the statement follows.
\end{proof}

We proceed to the proof of Lemma \ref{lm:cntexmpl_main}.

\begin{proof}[Proof of Lemma \ref{lm:cntexmpl_main}]

Observe that:
\begin{equation}
\label{eq:cntexmpl_decomp}
S_{R_1,R_2}f_{\lambda} = T_{R_1,R_2}f_{\lambda}-T_{-R_1,R_2}f_{\lambda}-T_{R_1,-R_2}f_{\lambda}+T_{-R_1,-R_2}f_{\lambda}, 
\end{equation}
where
\[T_{r_1,r_2}f(x_1,x_2) = -\frac{1}{4\pi^2} \operatorname{p.v.}\int_{\R^2}\frac{e^{2\pi \mathbbm{i}(r_1x_1'+r_2x_2')}}{x_1'x_2'}f(x_1-x_1',x_2-x_2') \,\textup{d}x_1'\textup{d}x_2'.\]

We prove the following two observations for the part (a) of the lemma. 

\begin{itemize}
\item[(1)] There exists $C>0$ such that for $\lambda$ large enough and $x_1,x_2\in [0,1]$:
    \[|{T_{\lambda x_2,\lambda x_1}f_{\lambda}(x_1,x_2)}|\ge C\log \lambda.\]
\item[(2)] For $\lambda>0$ and $x_1,x_2 \in [2/3,1]$, all of the expressions 
\[
|{T_{-\lambda x_2,\lambda x_1}f_{\lambda}(x_1,x_2)}|, \quad |{T_{\lambda x_2,-\lambda x_1}f_{\lambda}(x_1,x_2)}|, \quad |{T_{-\lambda x_2,-\lambda x_1}f_{\lambda}(x_1,x_2)}| 
\]
are bounded by a constant independent of $\lambda$.
\end{itemize}

In order to prove the observation (1), we note that because 
\[x_2x_1'+x_1x_2' + (x_1-x_1')(x_2-x_2')=x_1x_2+x_1'x_2',\] 
the following holds
\[|{T_{\lambda x_2,\lambda x_1}f_{\lambda}(x_1,x_2)}| = \frac{1}{4\pi^2} \bigg|{\operatorname{p.v.} \int_{[x_1-2,x_1+2]}\int_{[x_2-2,x_2+2]}
\frac{e^{2\pi \mathbbm{i} \lambda x_1'x_2'}}{x_1'x_2'}
\,\textup{d}x_2'\textup{d}x_1'}\bigg|.\]

We decompose the area of integration into 4 regions: 
\[ [-1,1]^2,\quad [-1,1]\times(\mathbb{R}\setminus[-1,1]),\quad (\mathbb{R}\setminus[-1,1])\times [-1,1],\quad (\mathbb{R}\setminus[-1,1])^2. \]
By the first part of Lemma \ref{lm:cntexmpl_crucial}, there exists $C>0$ such that the integral over the first region is at least $C\log \lambda$ whenever $\lambda$ is large enough. Integrals over the second and third regions are all $O(1)$ because of the following calculation:
\[\bigg|{\int_{1}^{x_2+2}\int_{-1}^{1} \frac{\sin(2\pi\lambda x_1'x_2')}{x_1'x_2'} \,\textup{d}x_1'\textup{d}x_2'}\bigg| = \int_{1}^{x_2+2} \frac{1}{x_2'} \bigg|{\int_{-2\pi\lambda x_2'}^{2\pi\lambda x_2'} \frac{\sin t}{t} \,\textup{d}t}\bigg| \,\textup{d}x_2' \lesssim \int_{1}^{3} \frac{1}{x_2'} \,\textup{d}x_2' \lesssim 1.\]
Finally, the integral over the last region is bounded using the triangle inequality by:
\[\int_{x_1-2}^{x_1+2}\int_{x_2-2}^{x_2+2} \frac{1}{x_1'x_2'}\1_{\{|x_1'|, |x_2'|>1\}}\,\textup{d}x_1'\textup{d}x_2' \lesssim 1.\]
Summing all the bounds we prove the observation (1).

We turn to the proof of the observation (2). First note that for $\epsilon_1,\epsilon_2\in\{-1,1\}$,
\begin{align*}
    &|{T_{\epsilon_1\lambda x_2,\epsilon_2\lambda x_1}f_{\lambda}(x_1,x_2)}|\\ 
    &= \frac{1}{4\pi^2} \bigg|{\operatorname{p.v.} \int_{[x_1-2,x_1+2]}\int_{[x_2-2,x_2+2]}
\frac{e^{2\pi \mathbbm{i}\lambda (x_1'x_2' + (\epsilon_1-1)x_1'x_2 + (\epsilon_2-1)x_2'x_1})}{x_1'x_2'}
\,\textup{d}x_2'\textup{d}x_1'}\bigg|.
\end{align*}
Assume, without loss of generality that $\epsilon_1=-1$. From the assumption on $x_2$ it follows that $|{(\epsilon_1-1)x_2}| \ge {4}/{3}$, so using the second part of Lemma \ref{lm:cntexmpl_crucial}, the integral over the first region is bounded by a constant. Integral over the fourth region is bounded as in the observation (1). Integrals over the second and the third region can be bounded using the following calculation
\[\bigg|{\int_{1}^{x_2+2}\int_{-1}^{1} \frac{\sin (2\pi \lambda x_1'(x_2'+(\epsilon_1 - 1)x_2))}{x_1'}\,\textup{d}x_1' \frac{e^{2\pi \mathbbm{i}\lambda(\epsilon_2-1)x_2'x_1}}{x_2'}\,\textup{d}x_2'}\bigg|\lesssim \int_{1}^{3}\frac{\textup{d}x_2'}{x_2'}\lesssim 1.  \]

Combining observations (1) and (2) with \eqref{eq:cntexmpl_decomp}, we conclude the proof of the part (a) of the lemma.

For the part (b), we observe that for $\epsilon_1,\epsilon_2\in\{-1,1\}$ the following holds:
\begin{align*}
& |{T_{\epsilon_1\lambda'x_2,\epsilon_2\lambda'x_1}f_{\lambda}(x_1,x_2)}| \\
& = \frac{1}{4\pi^2}\bigg|{\operatorname{p.v.}\int_{[x_1-2,x_1+2]}\int_{[x_2-2,x_2+2]} \frac{e^{2\pi \mathbbm{i}\lambda ( x_1'x_2' + ({\epsilon_1\lambda'}/{\lambda}-1)x_1'x_2+ ({\epsilon_2\lambda'}/{\lambda}-1)x_2'x_1)}}{x_1'x_2'} \,\textup{d} x_2'\textup{d}x_1'}\bigg|.
\end{align*}
We then decompose the area of integration in the same four parts as before. For the first part, since $|{(\epsilon_1\lambda'/\lambda-1)x_2}|\ge 4/3$, we use the second part of Lemma \ref{lm:cntexmpl_crucial} to get the upper bound and we treat the other parts as in the part (a) of the lemma. 
\end{proof}

\begin{remark}
It is obvious that the function $f$ in the proof of the part (b) is in $\textup{L}^1(\R^d)$, so the function $\widecheck{f}$, for which the convergence \eqref{eq:Fourier_conv} fails, is also continuous and therefore the counterexample exists in the class $C(\R^d)\cap \textup{L}^2(\R^d)$.
\end{remark}


\section{Proof of Corollary~\ref{cor:Strichartz}}
Let
\[ S := \Big\{\Big(\xi, \frac{\phi(\xi)}{2\pi}\Big) : \xi\in\mathbb{R}^n\Big\} \subseteq \mathbb{R}^{n+1} \]
be the hypersurface naturally associated with \eqref{eq:dispPDE}. Equip $S$ with the projection measure $\textup{d}\sigma(\xi,\tau)=\textup{d}\xi$. 
For every $g\in\textup{L}^2(S,\sigma)$ there exist a unique $f\in\textup{L}^2(\mathbb{R}^n)$ such that
\begin{equation}\label{eq:fgsubstit}
g\Big(\xi, \frac{\phi(\xi)}{2\pi}\Big) = \widehat{f}(\xi)
\end{equation}
for a.e.\@ $\xi\in\mathbb{R}^n$.
By the assumption \eqref{eq:Strichartz} and the Plancherel identity we then know that $\mathcal{E}$ given by the formula
\[ (\mathcal{E}g)(x,t) := (e^{\mathbbm{i} t \phi(D)} f)(x) \]
extends to a bounded linear operator $\mathcal{E}\colon\textup{L}^2(S,\sigma)\to\textup{L}^s(\mathbb{R}^{n+1})$.
In the case when $f\in\mathcal{S}(\mathbb{R}^{n})$, we can write
\[ (\mathcal{E}g)(x,t) 
= \int_{\mathbb{R}^n} e^{\mathbbm{i}t \phi(\xi) + 2\pi\mathbbm{i}x\cdot \xi} g\Big(\xi, \frac{\phi(\xi)}{2\pi}\Big) \,\textup{d}\xi
= \int_S e^{2\pi\mathbbm{i}(x,t)\cdot(\xi,\tau)} g(\xi,\tau) \,\textup{d}\sigma(\xi,\tau) \]
and, taking another Schwartz function $h\in\mathcal{S}(\mathbb{R}^{n+1})$,
\[ \int_{\mathbb{R}^{n+1}} h(x,t) \,\overline{(\mathcal{E}g)(x,t)} \,\textup{d}x \,\textup{d}t = \int_{S} \widehat{h}(\xi,\tau) \,\overline{g(\xi,\tau)} \,\textup{d}\sigma(\xi,\tau). \]
By duality we now see that the a priori restriction estimate \eqref{eq:restr_apriori} holds with $d=n+1$, $p=s'$, $q=2$.
In fact, the Fourier restriction operator $\mathcal{R}\colon\textup{L}^{s'}(\mathbb{R}^{n+1})\to\textup{L}^2(S,\sigma)$
is now known to be bounded and its adjoint is precisely $\mathcal{E}$, which is for this reason sometimes called the \emph{Fourier extension operator}.

Note that $p=s'<2=q$. Now Theorem~\ref{thm:meainrestr} applies, so that the maximal estimate \eqref{eq:restr_multimax} gives
\begin{equation}\label{eq:restr_multimax2}
\Big\|\sup_{r_1,\ldots,r_{n+1}\in(0,\infty)}\big|\widehat{h}\ast\chi_{r_1,\ldots,r_{n+1}}\big|\Big\|_{\textup{L}^2(S,\sigma)} \lesssim_{n,\phi,\chi,s} \|h\|_{\textup{L}^{s'}(\mathbb{R}^{n+1})}
\end{equation}
for any given Schwartz function $\chi\in\mathcal{S}(\mathbb{R}^{n+1})$.
If we extend the definition of dilates as
\[ \chi_{r_1,\ldots,r_d}(x_1,\ldots,x_d) := \frac{1}{|r_1\cdots r_d|} \chi\Big(\frac{x_1}{r_1},\ldots,\frac{x_d}{r_d}\Big) \]
for $r_1,\ldots,r_d\in\mathbb{R}\setminus\{0\}$, then \eqref{eq:restr_multimax2} implies
\begin{equation}\label{eq:restr_multimax3}
\Big\|\sup_{r_1,\ldots,r_{n+1}\in\mathbb{R}\setminus\{0\}}\big|\widehat{h}\ast\chi_{r_1,\ldots,r_{n+1}}\big|\Big\|_{\textup{L}^2(S,\sigma)} \lesssim_{n,\phi,\chi,s} \|h\|_{\textup{L}^{s'}(\mathbb{R}^{n+1})},
\end{equation}
by considering $2^{n+1}$ quadrants of $\mathbb{R}^{n+1}$, flipping $\chi$ as necessary, and increasing the implicit constant by the factor $2^{n+1}$.
Linearizing and dualizing \eqref{eq:restr_multimax3} we obtain
\[ \bigg| \int_S \big(\,\widehat{h}\ast\chi_{r_1(\xi),\ldots,r_{n+1}(\xi)}\big)(\xi,\tau) \,\overline{g(\xi,\tau)} \,\textup{d}\sigma(\xi,\tau) \bigg| \lesssim_{n,\phi,\chi,s} \|g\|_{\textup{L}^2(S,\sigma)} \|h\|_{\textup{L}^{s'}(\mathbb{R}^{n+1})} \]
for any choice of measurable functions $r_1,\ldots,r_{n+1}\colon\mathbb{R}^n\to\mathbb{R}\setminus\{0\}$.
If we further substitute \eqref{eq:fgsubstit} and choose $\chi$ such that $\widecheck{\chi}=\overline{\psi}$, then we can rewrite the last bilinear estimate as
\[ \bigg| \int_{\mathbb{R}^{n+1}} h(x,t) \,\overline{(T_{\psi,r_1,\ldots,r_{n+1}}f)(x,t)} \,\textup{d}x \,\textup{d}t \bigg| \lesssim_{n,\phi,\psi,s} \|f\|_{\textup{L}^2(\mathbb{R}^n)} \|h\|_{\textup{L}^{s'}(\mathbb{R}^{n+1})}, \]
which is just the dualized formulation of the desired bound \eqref{eq:maxStrichartz}.
The case of general measurable functions $r_1,\ldots,r_{n+1}\colon\mathbb{R}^n\to\mathbb{R}$ now easily follows in the limit.


\section*{Acknowledgments}
This work was supported in part by the \emph{Croatian Science Foundation} project UIP-2017-05-4129 (MUNHANAP).


\bibliography{maximal_restriction}{}
\bibliographystyle{plain}

\end{document}